\documentclass[12pt,a4paper]{article}
\usepackage{amssymb}
\usepackage{amsmath}
\usepackage{amsthm}
\usepackage{indentfirst}
\usepackage{ot1patch}
\usepackage{psfrag}
\usepackage{rotate}
\usepackage{mathtools}
\usepackage{hyperref}

\newtheorem{tw}{Theorem}[section]
\newtheorem{pr}[tw]{Proposition}
\newtheorem{lm}[tw]{Lemma}
\newtheorem{cor}[tw]{Corollary}

\theoremstyle{definition}

\newtheorem{ex}[tw]{Example}

\DeclareMathOperator{\Irr}{Irr}
\DeclareMathOperator{\Sqf}{Sqf}

\title{On properties of composites and monoid domains}

\author{\L ukasz Matysiak\\
Kazimierz Wielki University\\
Bydgoszcz, Poland \\
lukmat@ukw.edu.pl}
\title{On properties of composites and monoid domains}

\begin{document}

\maketitle

\begin{abstract}
In this paper I consider all possible properties from commutative algebra for polynomial composites and monoid domains. The aim is full characterization of these structures. I start with the examination of group, ring, modules properties, graded, but also the study of invertible elements, irreducible elements, ideals, etc. in these structures.
In the second part of the work I give examples of the use of composites and monoid domains in cryptology. Each such polynomial is the sum of the products of the variable and the coefficient. And what if subsequent coefficient sets are appropriate cryptographic systems? Similarly, monoid domains can be a very good tool between encrypting and decrypting messages. 
\end{abstract}

\begin{table}[b]\footnotesize\hrule\vspace{1mm}
\end{table}

\section{Introduction}

By a ring we mean a commutative ring with unity. Let $R$ be a ring.
We denote by $R^{\ast}$ the group of all invertible elements of $R$. The set of all irreducible elements in $R$ will be denoted by $\Irr R$.
Recall that a GCD-domain is an integral domain $R$ with the property that any two elements have a a greatest common divisor (GCD); i.e., there is a unique minimal principal ideal containing the ideal generated by given two elements. We can find much informations about GCD-domains. I refer to \cite{xx1}, \cite{8}, \cite{3}.

\medskip

The main motivation of this paper is description of some algebraic objects in the language of commutative algebra. D.D.~Anderson, D.F.~Anderson, M. Zafrullah in \cite{1} called object $A+XB[X]$ as a composite for $A\subset B$ fields. If $B$ is a domain and $M$ is an additive cancellative monoid we can define a monoid domain $B[M]=\{a_0X^{m_0}+\dots +a_nX^{m_n}: a_0, \dots , a_n\in B, m_1, \dots m_n\in M\}$. Monoid domains appear in many works such that \cite{2}, \cite{9}.

\medskip

There are a lot of works where composites are used as examples to show some properties. But the most important works are presented below.

\medskip

In 1976 \cite{y1} authors considered the structures in the form $D+M$, where $D$ is a domain and $M$ is a maximal ideal of ring $R$, where $D\subset R$. Later (\ref{t2}), we could prove that in composite in the form $D+XK[X]$, where $D$ is a domain, $K$ is a field with $D\subset K$, that $XK[X]$ is a maximal ideal of $K[X]$. 
Next, Costa, Mott and Zafrullah (\cite{y2}, 1978) considered composites in the form $D+XD_S[X]$, where $D$ is a domain and $D_S$ is a localization of $D$ relative to the multiplicative subset $S$. 
In 1988 \cite{y5} Anderson and Ryckaert studied classes groups $D+M$.
Zafrullah in \cite{y3} continued research on structure $D+XD_S[X]$ but 
he showed that if $D$ is a GCD-domain, then the behaviour of $D^{(S)}=\{a_0+\sum a_iX^i\mid a_0\in D, a_i\in D_S\}=D+XD_S[X]$ depends upon the relationship between $S$ and the prime ideals $P$ od $D$ such that $D_P$ is a valuation domain (Theorem 1, \cite{y3}).
Fontana and Kabbaj in 1990 (\cite{y4}) studied the Krull and valuative dimensions of composite $D+XD_S[X]$. 
In 1991 there was an article (\cite{1}) that collected all previous composites and the authors began to create a theory about composites creating results. In this paper, the structures under consideration were officially called as composites. 
After this article, various minor results appeared. But the most important thing is that composites have been used in many theories as examples. That is why I decided to examine all possible properties of composites for commutative algebra. I put the first results in \cite{5}, and the next results I put in this article.

\medskip

In the second chapter I will present a more general concept of a composite defining two types and $I(B,A)$, and monoid domain $F[X;M]$, with different configurations of rings, domains, fields, etc. We will show some simple properties from commutative algebra for these structures. In the Lemma \ref{p1} it will turn out that one of the types of composites together with the multiplicative action can not be a semigroup. In the Theorem \ref{t2} we show that every nonzero prime ideal in the composite is the maximal ideal. We will also see what are the irreducible elements of such composites. In the Theorem \ref{t3} we will look at what elements are irreducible in the considered structures. At the end of the chapter, we will look at the multiplicative systems of these structures.

\medskip

In the third chapter I present simple observations by building on $R$-modules from the considered structures (Proposition \ref{prpr3}). I also asked a question about exact sequence for composites. The next chapter is devoted to locations in composites. In the fifth chapter, I mention graded rings and modules.

\medskip

The main motivation of sixth chapter is description composites and monoid domains in language for the monoids/gropus (for composites) and for rings (for monoid rings).

The structures described in this paper often appear in the form of examples in many works. The aim of this work will be to examine as much as possible properties and applications in commutative algebra.

\medskip

In last chapters of this paper I give an example of the use of composites and monoid domains in cryptology.

\section{Basic properties}
\label{S1}

The aim of this chapter will be to examine the simplest structural properties of the considered structures.

Consider $A$ and $B$ as rings such that $A\subset B$.
Put $T=A+XB[X]$. The structure defined in this way is called a composite. (The definition comes from \cite{1}). Later results will require more assumptions.

\medskip

\noindent
Let us generalized the concept of a composite in two different directions.

\medskip

Consider $A_0, A_1, \dots , A_{n-1}$ and $B$ be rings for any $n \geq 0$ such that $A_0\subset A_1\subset \dots \subset A_{n-1}\subset B$.
Put $T_n=A_0+A_1X+\dots + A_{n-1}X^{n-1} + X^nB[X]$.

\medskip

\noindent
And let other $A_0, A_1, \dots , A_{n-1}$ and $B$ be rings for any $n \geq 0$ such that there exists $i\in\{0, 1, \dots , n-1\}$, where $A_i\not\subset A_{i+1}$ and for every $j\in\{0, 1, \dots , n-1\}$ we have $A_j\subset B$.
Put $T'_n=A_0+A_1X+\dots + A_{n-1}X^{n-1} + X^nB[X]$.

\medskip

Also consider $I(B,A)=\{f\in B[X], f(A)\subseteq A\}$, where $A\subset B$ are rings and monoid domain $B[M]$, where $B$ be a domain and $M$ be a submonoid of $\mathbb{Q}_+$.

\medskip

Throughout the article, we will use the above denotation. If the above structures require stronger assumptions, then of course this will be included.

\medskip

It is easy to check the following Lemmas.

\medskip

\begin{lm}
	\label{l1}
	$T_n$ is a ring, $T_n\subset T$ and $T'_n\subset T$.
\end{lm}

\begin{lm}
	\label{p1}
	If $f, g\in T'_n$, where $n>0$, then $fg\in A_0+XB[X]$.
\end{lm}

\begin{cor}
	\label{c1}
	$(T'_n, \cdot)$ is not a semigroup. 
\end{cor}

\begin{lm}
	$I(B,A)$ is a ring.
\end{lm}

Now let's look at invertible and nilpotent elements.

\begin{pr}
	\label{p2}
	Let $f=a_0+a_1X+\dots + a_nX^n\in T$ for any $n\geq 0$. Then $f\in T^{\ast}$ if and only if $a_0\in A^{\ast}$ and $a_1, a_2, \dots , a_n$ are nilpotents.
\end{pr}

\begin{proof}
	We know that if $R$ is a ring then $f=a_0+a_1X+\dots + a_nX^n\in R[X]^{\ast}$ if and only if $a_0\in R^{\ast}$ and $a_1, a_2, \dots , a_n$ are nilpotents. In our Proposition we have $a_1, a_2, \dots , a_n$ are nilpotents. Of course we get $a_0\in A^{\ast}$.
\end{proof}

\begin{pr}
	\label{p3}
	Let $f=a_0+a_1X+\dots a_{n-1}X^{n-1}+a_nX^n+\dots +a_mX^m\in T_n$, ($T'_n$, respectively), where $0\leq n\leq m$ and $a_i\in A_i$ for $i=0, 1, \dots , n$ and $a_j\in B$ for $j=n, n+1, \dots , m$.
	\begin{itemize}
		\item[(i) ] $f\in T_n^{\ast}$, (${T'_n}^{\ast}$, respectively) if and only if $a_0\in A_0^{\ast}$ and $a_1, a_2, \dots , a_m$ are nilpotents.
		\item[(ii) ] $f$ is a nilpotent if and only if $a_0, a_1, \dots , a_m$ are nilpotents.
	\end{itemize}
\end{pr}

\begin{proof}
	Analogous proof like in Proposition \ref{p2}.
\end{proof}

\begin{pr}
	Let $f=a_0+a_1X+\dots + a_nX^n\in I(B,A)$, where $A\subset B$ are domains. 
	\begin{itemize}
		\item[(i) ] $f\in I(B,A)^{\ast}$ if and only if $a_0\in A^{\ast}$ and $a_1, a_2, \dots , a_n$ are nilpotents.
		\item[(ii) ] $f$ is nilpotent if and only if $a_0, a_1, \dots , a_n$ are nilpotents.
	\end{itemize}
\end{pr}

\begin{pr}
	\label{prpr1}
	Let $B$ be a domain and $f=a_{m_1}X^{m_1}+a_{m_2}X^{m_2}+\dots +a_{m_n}X^{m_n}\in B[M]$, where $m_1, m_2, \dots , m_n\in M$ and $a_{m_1}, a_{m_2}, \dots , a_{m_n}\in B$. 
	\begin{itemize}
		\item[(i) ] $f\in B[M]^{\ast}$ if and only if there exist $m_i\in M$ and $a_{m_i}\in B^{\ast}$ such that $m_i=0$ and for every $k\neq m$ we have $a_k$ be nilpotents.
		\item[(ii) ] $f$ be a nilpotent if and only if $a_{m_1}, a_{m_2}, \dots , a_{m_n}$ are nilpotents.
	\end{itemize}
\end{pr}

\begin{proof}
	$(i)$ Assume $f\in B[M]^{\ast}$. Then there exists $g=b_{m'_1}X^{m'_1}+b_{m'_2}X^{m'_2}+\dots +b_{m'_n}X^{m'_n}$, where $m'_1, m'_2, \dots , m'_n\in M$ and $b_{m'_1}, b_{m'_2}, \dots , b_{m'_n}\in B$ such that $fg=1$. Hence there exist $m_i, m_j\in M$ such that $a_{m_i}b_{m_j}X^{m_i+m_j}=1$. We have $a_i\in B^{\ast}$ and $m_i, m_j=0$. The rest of coefficients are nilpotents. On the other side of the proof it is easy.
	
	\medskip
	
	\noindent
	$(ii)$ Obvious.
\end{proof}

Let's recall Theorem from \cite{1} (Theorem 2.9) in a different form.

\begin{tw}
	\label{t1}
	Let $A$ be a subfield of $B$. Consider $D=A+XB[X]$. Then 
	$\Irr D=\{ax, a\in B\}\cup
	\{a(1+xf(x)),a\in A, f\in B[x], 1+xf(x)\in\Irr B[x]\}.$
\end{tw}

\begin{tw}
	\label{t2}
	Consider $T=A+XB[X]$, where $A$ be a subfield of $B$; $T_n=A_0+A_1X+A_2X^2\dots +A_{n-1}X^{n-1}+X^nB[X]$, where $A_0\subset A_1\subset A_2\subset\dots \subset A_{n-1}\subset B$ be fields. Then
	\begin{itemize}
		\item[(i) ] every nonzero prime ideal of $T$ ($T_n$, respectively) is maximal;
		\item[(ii) ] every prime ideal $P$ different from $XB[X]$ (in $T$) is principal;
		\item[(iii) ] every prime ideal $P$ different from $A_1X+A_2X^2+\dots + A_{n-1}X^{n-1}+X^nB[X]$ (in $T_n$) is principal;
		\item[(iv) ] $T$ is atomic, i. e., every nonzero nonunit of $T$ is a finite product of irreducible elements (atoms);
		\item[(v) ] $T_n$ is atomic.
	\end{itemize} 
\end{tw}

\begin{proof}
	\textit{(i)}. For $T$ we have in \cite{1}, Theorem 2.9 \textit{(i)}. We proof for $T_n$.
	
	\medskip
	
	\noindent
	First note that $A_1X+A_2X^2+\dots + A_{n-1}X^{n-1}+X^nB[X]$ is maximal since $T_n/A_1X+A_2X^2+\dots + A_{n-1}X^{n-1}+X^nB[X]\cong A_0$. Let $P$ be a nonzero prime ideal of $T_n$. Now $X\in P$ implies $(T_n/A_1X+A_2X^2+\dots + A_{n-1}X^{n-1}+X^nB[X])^2\subseteq P$ and hence $A_1X+A_2X^2+\dots + A_{n-1}X^{n-1}+X^nB[X]\subseteq P$ so $P=A_1X+A_2X^2+\dots + A_{n-1}X^{n-1}+X^nB[X]$. So suppose that $X\notin P$. Then for $N=\{1, X, X^2, \dots \}$, $P_N$ is a prime ideal in the PID $B[X,X^{-1}]=T_{n,N}$. (In fact, $B[X,X^{-1}]\subseteq R_P$ and $R_P$ is a DVR (discrete valuation ring).) So $P$ is minimal and is also maximal unless $P\subsetneq A_1X+A_2X^2+\dots + A_{n-1}X^{n-1}+X^nB[X]$. But let $k_nX^n+\dots + k_sX^s\in P$ with $k_n\neq 0$, where $k_n, \dots , k_s\in B$ for any $n, s$. Then $X^{n+1}+k_n ^{-1}k_{n+1}X^{n+2}+\dots +k_n ^{-1}k_sX^s\in P$, so $X\notin P$ implies that $1+k_n ^{-1}k_{n+1}X+\dots +k_n ^{-1}k_sX^{s-n}\in P$, a contradiction. So every nonzero prime ideal is maximal.
	
	\medskip
	
	\noindent
	\textit{(ii)}. \cite{1}, Theorem 2.9 \textit{(ii)}.
	
	\medskip
	
	\noindent
	\textit{(iii)}. If $P$ is different from $A_1X+A_2X^2+\dots + A_{n-1}X^{n-1}+X^nB[X]$, then it contains an element of the form $1+a_1X+a_2X^2+\dots +a_{n-1}X^{n-1}+X^nf(X)$, where $a_i\in A_i$ for $i=1, 2, \dots , n-1$ and $f(X)\in B[X]$. Now if $1+a_1X+a_2X^2+\dots +a_{n-1}X^{n-1}+X^nf(X)$ can be factored in $A_1X+A_2X^2+\dots + A_{n-1}X^{n-1}+X^nB[X]$ it can be written as $(1+b_1X+b_2X^2+\dots +b_{n-1}X^{n-1}+X^ng(X))(1+c_1X+c_2X^2+\dots +c_{n-1}X^{n-1}+X^nh(X))$, where $b_i, c_i\in A_i$ for $i=1, 2, \dots , n-1$ and $g(X), h(X)\in B[X]$. Hence $1+a_1X+a_2X^2+\dots +a_{n-1}X^{n-1}+X^nf(X)$ is irreducible in $T_n$ if and only if it is irreducible in $A_1X+A_2X^2+\dots + A_{n-1}X^{n-1}+X^nB[X]$. 
	
	\medskip
	
	\noindent
	Now let $1+a_1X+a_2X^2+\dots +a_{n-1}X^{n-1}+X^nf(X)$ be irreducible in $T_n$ and suppose that $1+a_1X+a_2X^2+\dots +a_{n-1}X^{n-1}+X^nf(X)\mid k(X)l(X)$ in $T_n$. Then $1+a_1X+a_2X^2+\dots +a_{n-1}X^{n-1}+X^nf(X)\mid k(X)l(X)$ in $A_1X+A_2X^2+\dots + A_{n-1}X^{n-1}+X^nB[X]$, and so in $A_1X+A_2X^2+\dots + A_{n-1}X^{n-1}+X^nB[X]$ we have, say $1+a_1X+a_2X^2+\dots +a_{n-1}X^{n-1}+X^nf(X)\mid k(X)$. Then, in $A_1X+A_2X^2+\dots + A_{n-1}X^{n-1}+X^nB[X]$, $k(X)=(1+a_1X+a_2X^2+\dots +a_{n-1}X^{n-1}+X^nf(X))d(X)$. Now $d(X)$ can be written as $d(X)=aX^r(1+a_1X+a_2X^2+\dots +a_{n-1}X^{n-1}+X^np(X))$. If $r>0, d(X)\in T_n$, while if $r=0, k(X)=(1+a_1X+a_2X^2+\dots +a_{n-1}X^{n-1}+X^nf(X))(b(1+b_1X+b_2X^2+\dots +b_{n-1}X^{n-1}+X^np(X))$ and $b\in A_0$ because $k(X)\in T_n$. In either case, $d(X)\in T_n$ and so $1+a_1X+a_2X^2+\dots +a_{n-1}X^{n-1}+X^nf(X)\mid k(X)$ in $T_n$. Consequently, in $T_n$ every irreducible element of the type $1+a_1X+a_2X^2+\dots +a_{n-1}X^{n-1}+X^nf(X)$ is prime.
	
	\medskip
	
	\noindent
	Now since every element of the form $1+a_1X+a_2X^2+\dots +a_{n-1}X^{n-1}+X^nf(X)$ is a product of irreducible elements of the same form and hence is a product of prime elements, it follows that every prime ideal of $P$ different from $A_1X+A_2X^2+\dots + A_{n-1}X^{n-1}+X^nB[X]$ contains a principal prime and hence is actually principal.
	
	\medskip
	
	\noindent
	\textit{(iv)}. \cite{1}, Theorem 2.9 \textit{(iii)}.
	
	\medskip
	
	\noindent
	\textit{(v)}. From \textit{(iii)} a general element of $T_n$ can be written as $aX^r(1+a_1X+a_2X^2+\dots +a_{n-1}X^{n-1}+X^nf(X))$, where $a\in B$ (with $a\in A_0$ if $r=0$) and $1+a_1X+a_2X^2+\dots +a_{n-1}X^{n-1}+X^nf(X)$ is a product of primes.
\end{proof}

\begin{lm}
	If $A\subset B$ be fields and $B[X]$ be a GCD-domain then $I(B,A)$ be a GCD-domain.
\end{lm}

\begin{ex}
	\label{c2}
	$T, T_n$ are no GCD-domains. Let $f=a_1+b_1X, g=a_2+b_2X$, where $a_1, a_2\in A, b_1, b_2\in B$ with $A+XB[X]$. Then $\gcd(f, g)=\dfrac{a_1b_2-a_2b_1}{b_2}$. We see that $\gcd (f,g)\in B\setminus A$.
\end{ex}

Recall that a domain $R$ is a pre-Schreier domain if every element $a\in R$ is a primal, i.e. for every elements $b, c\in H$ if $a\mid bc$ then there exist $a_1, a_2\in R$ such that $a_1\mid b, a_2\mid c, a=a_1a_2$. 

More information about Schreier and pre-Schreier domains we can see in many works, e.g. in \cite{4}, \cite{7}, \cite{10}, \cite{3}, \cite{5}, respectively.

\begin{lm}
	If $A\subset B$ be fields, then $T$ be a pre-Schreier domain. If $A_0\subset A_1\subset \dots A_{n-1}\subset B$ be fields, then $T_n$ is also pre-Schreier domain. But $T'_n$ doesn't satisfies a pre-Schreier conditions.
\end{lm}

The following lemma states the uniqueness of two polynomials in the composites $T_n$ and $T_n'$.

\begin{lm}
	\label{c4}
	Let $f=a_0+a_1X+a_2X^2+\dots + a_nX^n + \dots + a_rX^r, g=b_0+b_1X+b_2X^2+\dots b_nX^n+\dots + b_sX^s$, where $r\geq n\geq 0, s\geq n\geq 0$ and $a_i, b_i\in A_i$ for $i=0, 1, 2, \dots , n$ and $a_j\in B$ for $j=n+1, n+2, \dots , r$ and $b_j\in B$ for $j=n+1, n+2, \dots , s$. If $f=g$ then $r=s$ and $a_i= b_i$ for $i=0, 1, 2, \dots , r$.
\end{lm}

\begin{proof}
	Let $f=a_0+a_1X+a_2X^2+\dots + a_nX^n + \dots + a_rX^r, g=b_0+b_1X+b_2X^2+\dots b_nX^n+\dots + b_sX^s$, where $s\geq r\geq n\geq 0$ and $a_i, b_i\in A_i$ for $i=0, 1, 2, \dots , n$ and $a_j\in B$ for $j=n+1, n+2, \dots , r$ and $b_j\in B$ for $j=n+1, n+2, \dots , s$. 
	Consider $a_0+a_1X+a_2X^2+\dots + a_nX^n + \dots + a_rX^r = b_0+b_1X+b_2X^2+\dots b_nX^n+\dots + b_sX^s$. Then $(b_0-a_0)+(b_1-a_1)X+(b_2-a_2)X^2+\dots + (b_n-a_n)X^n + \dots + (b_r-a_r)X^r + b_{r-1} + \dots + b_sX^s=0$. Hence $r=s$ and $a_i= b_i$ for $i=0, 1, 2, \dots , r$.
\end{proof}

\begin{tw}
	\label{t3}
	Consider $T=A+XB[X]$, where $A$ be a subfield of $B$; $T_n=A_0+A_1X+A_2X^2\dots +A_{n-1}X^{n-1}+X^nB[X]$, where $A_0\subset A_1\subset A_2\subset\dots \subset A_{n-1}\subset B$ be fields; $I(B,A)$, where $A\subset B$ are domains. Then
	\begin{itemize}
		\item[(i) ] $f\in\Irr T$ if and only if $f\in\Irr B[X], f(0)\in A$.
		\item[(ii) ] $f\in\Irr T_n$ if and only if $f\in\Irr B[X], a_i\in A_i$, where $f=a_0+a_1X+\dots a_{n-1}X^{n-1}+a_nX^n+\dots +a_{m}X^m$ with $a_i\in A_i$ for $i=0, 1, \dots n-1$ and $a_n, a_{n+1}, \dots , a_m\in B (n<m)$.
		\item[(iii) ] $f\in\Irr I(B,A)$ if and only if $f\in\Irr B[X]$ and $f(A)\subseteq A$.
	\end{itemize}
\end{tw}

\begin{proof}
	\textit{(i)}. 
	Suppose that $f\notin\Irr B[X]$ or $f(0)\notin A$.
	If $f(0)\notin A$, then $f\notin T$, so $f\notin\Irr B[X]$.
	Now, assume that $f\notin\Irr B[X]$.
	Then $f=gh$, where $g, h\in B[x]\setminus B$.
	Let $g=a_0+a_1X+\dots +a^nX^n, h=b_0+b_1X+\dots + b_mX^m$.
	We have $f=(a_0+a_1X+\dots +a^nX^n)(b_0+b_1X+\dots + b_mX^m)$.
	Then
	$f=\big(1+\dfrac{a_1}{a_0}X+\dots + \dfrac{a_n}{a_0}X^n\big)
	(a_0b_0+a_0b_1X+\dots +a_0b_mX^m)$,
	where $a_0b_0=f(0)\in A$.
	Now, suppose that $f\notin\Irr T$.
	If $f\notin T$, then $f(0)\notin A$.
	Now, assume that $f\in T$.
	Then we have $f=gh$, where $g, h\in T\setminus A$.
	This implies $g, h\in B[x]\setminus B$.
	
	\medskip
	
	\noindent
	\textit{(ii)} occur in the same way as in \textit{(i)}.
	
	\medskip
	
	\noindent
	\textit{(iv) } From definition $I(B,A)$ we have $f\in\Irr B[X]$ and $f(A)\subseteq A$. 
\end{proof}

In \cite{5} Proposition 8.6 with J\c{e}drzejewicz, Marciniak and Zieli\'nski we received the characterization of square-free elements in composite $A+XB[X]$, for $A\subset B$ fields, i.e.

\begin{pr}
	Let $L$ and $F$ be fields such that $L\subset F$ and let $T=L+XF[X]$. Then
	$\Sqf T = (\Sqf F[X]\cap T)\cup \{x^2h;h\in \Sqf F[x], h(0)\not\in \{a^2b;a\in F,b\in L\}\}$
\end{pr}

\begin{cor}
Let $A_0\subset A_1\subset\dots \subset A_{n-1}\subset B$ be fields. Then
$\Sqf T_n = (\Sqf B[X]\cap T_n)\cup \{x^2h;h\in \Sqf B[X], h(0)\not\in \{a^2b;a\in B,b\in A\}\}$.
\end{cor}

In \cite{2}, Lemma 6.4 we have informations about irreducible element in monoid domain $D[S]$, where $D$ be a domain, and $S$ be a submonoid of $\mathbb{Q}_+$ . I present a generalized Proposition.

\begin{pr}
	\label{x1}
	Let $B$ be an integral domain with quotient field $K$ and $M$ a monoid with quotient group $G\neq M$. Assume that $B$ contains prime elements $p_1, p_2, \dots , p_{r-1}$. Assume that $M$ is integrally closed and each nonzero element of $G$ is type $(0, 0, \dots )$ ($G$ satisfies the ascending chain condition on cyclic subgroups).
	Consider $m_1, m_2, \dots , m_r\in M$ such that $m_1\in\Irr M$ and $m_2, m_3, \dots , m_r\notin m_1 + M$. Then $p_{r-1}X^{m_{r}}-\dots -p_2X^{m_3}-p_1X^{m_2} - X^{m_1}\in\Irr B[M]$.
\end{pr}

\begin{proof}
	Let $\leq$ be a total order on $G$. We may assume that $m_r<m_{r-1}<\dots m_2<m_1$. Suppose that $p_{r-1}X^{m_{r}}-\dots -p_2X^{m_3}-p_1X^{m_2} - X^{m_1}=fg$ with $f,g\in B[M]$. Write $f=a_1X^{t_1}+\dots a_mX^{t_m}$ and $g=b_1X^{k_1}+\dots +b_nX^{k_n}$ in canonical form, where $t_1< \dots < t_m$ and $k_1< \dots < k_n$. First assume that either $f$ or $g$ is a monomial, say $f=aX^t$.Then $a\in B^{\ast}, m_1=t+k_n, m_2=t+k_1, m_3=t+k_2, \dots , m_r=t+k_{r-1}$. Since $m_1\in\Irr M$, either $t$ or $k_n$ is invertible in $M$. If $k_n$ is invertible, then $m_2=t+k_1=(m_1-k_n)+k_1\in m_1+M, m_3=t+k_2=(m_1-k_n)+k_2\in m_1+M, \dots , m_r\in m_1+M$, a contradiction. Thus $t$ is invertible in $M$, and hence $f$ is a unit in $B[M]$. Thus we may assume that $f$ and $g$ are not monomials. Now consider the reduction of $p_{r-1}X^{m_{r}}-\dots -p_2X^{m_3}-p_1X^{m_2} - X^{m_1}=fg$ modulo the ideal $(p_1, p_2, \dots , p_{r-1})$. Then $(-1+(p_1, p_2, \dots , p_{r-1})=((a_m+(p_1, p_2, \dots , p_{r-1}))X^{t_m})((b_n+(p_1, p_2, \dots , p_{r-1}))X^{k_n})$.  This means that $a_1+(p_1, p_2, \dots , p_{r-1})=b_1+(p_1, p_2, \dots , p_{r-1})=(p_1, p_2, \dots , p_{r-1})$. In this case $c_1p_1+\dots c_{r-1}p{r-1}=a_1b_1\in (p_1, \dots , p_{r-1})^2$, a contradiction. Thus $p_{r-1}X^{m_{r}}-\dots -p_2X^{m_3}-p_1X^{m_2} - X^{m_1}\in\Irr B[M]$. 
\end{proof}

Now, We give some basic information related to ideals.

\begin{pr}
	\label{p4}
	\begin{itemize}
		\item[(i) ] If $A$ be a field, then $XB[X]$ be an maximal ideal in $T$.
		\item[(ii) ] If $A$ be an integral domain, then $XB[X]$ be an prime ideal in $T$.
		\item[(iii) ] $T/(X)\cong A$.
		\item[(iv) ] $T/B\cong \{0\}$.
		\item[(v) ] Let $A\subset B$ be fields in $T$. $T/(aX)$ be a field for any $a\in B$.
		\item[(vi) ] Let $A\subset B$ be fields in $T$. $T/(a(1+Xf(X)))$ be a field for any $a\in A, f\in B[X]$ such that $1+Xf(X)\in\Irr B[X]$.
	\end{itemize}
\end{pr}

\begin{proof}
(i) Let $A$ be a field. The proof follows from $T/XB[X]\cong A$. We have $XB[X]$ is a maximal ideal in $T$.

\medskip

(ii) -- (iv) Obvious.

\medskip

(v), (vi) From Theorem 2.9 in \cite{1} $aX$ for any $a\in B$ is an irreducible element. We get $T/(aX)$ be a field. We also have $a(1+Xf(X))$ for any $a\in A, f\in B[X]$ such that $1+Xf(X)\in\Irr B[X]$ is a irreducible element. We have $T/(a(1+Xf(X)))$ be a field.
\end{proof}

\begin{pr}
	\label{p5}
	\begin{itemize}
		\item[(i) ] If $A_0+A_1X+\dots +A_{n-1}X^{n-1}$ be a field (where $A_0\subset A_1\subset A_2\subset\dots \subset A_{n-1}\subset B$), then $X^nB[X]$ be an maximal ideal in $T_n$.
		\item[(ii) ] If $A_0+A_1X+\dots +A_{n-1}X^{n-1}$ be a domain, then $X^nB[X]$ be an prime ideal in $T_n$.
		\item[(iii) ] $T_n/(X)\cong A_0$.
		\item[(iv) ] $T_n/B\cong \{0\}$.
		\item[(v) ] Let $A_0\subset A_1\subset\dots \subset B$ be fields in $T_n$. $T_n/(aX)$ be a field for any $a\in B$.
		\item[(vi) ] Let $A_0\subset A_1\subset\dots \subset B$ be fields in $T_n$. $T_n/(a(1+a_1X+a_2X^2+\dots +a_{n-1}X^{n-1}+ X^nf(X)))$ be a field for any $a\in B, a_i\in A_i (i=1, 2, \dots , n-1), f\in B[X]$ such that $1+Xf(X)\in\Irr B[X]$.
	\end{itemize}
\end{pr}

\begin{proof}
	This proof is similarly to proof of Proposition \ref{p4}.
\end{proof}

The Proposition \ref{p5} holds for $T'_n=A_0+A_1X+A_2X^2+\dots + A_{n-1}X^{n-1}+X^nB[X]$ where $A_0, A_1, A_2, \dots , A_{n-1}\subset B$, $A_i\not\subset A_j$ for $i\neq j$ be fields.

\begin{pr}
	\begin{itemize} 
		\item[(i) ] $I(B,A)/A=I(B,A)/B=\{0\}$.
		\item[(ii) ] $B[M]/(p_{r-1}X^{m_r}-\dots - p_1X^{m_2}-X^{m_1})$ be a field, where $B$ be a domain, $p_1, p_2, \dots , p_{m_r}\in B, m_1, m_2, \dots , m_r\in M$ with $m_1\in\Irr M, m_2, m_3, \dots ,$ 
		$m_r\notin m_1+M$.
	\end{itemize}
\end{pr}

\begin{proof}
	\begin{itemize}
		\item[(i) ] It follows from definition.
		\item[(ii) ] It follows from Proposition \ref{x1}.
	\end{itemize}
\end{proof}

The next statements give a description of some multiplicative systems. Proofs of these statements are easy.

\begin{pr}
	\label{p6}
	If $A$ be an integral domain then $S=A+XB[X]\setminus XB[X]$ is a saturated multiplicative system.
\end{pr}

\begin{pr}
	\label{p7}
	If $A_0+A_1X+\dots +A_{n-1}X^{n-1}$ be an integral domain. Then $S=T_n\setminus X^nB[X]$ (with appropriate assumptions) is saturated multiplicative system.
\end{pr}

\begin{cor}
	Consider $T_n$ where $A_0, A_1, \dots A_{n-1}, B$ be fields (with appropriate assumptions). Then $S={A_0}^{\ast}+A_1X+A_2X^2+\dots +A_{n-1}X^{n-1}+X^nB[X]$  be a saturated multiplicative system.
\end{cor}

\begin{pr}
	$I(B,A)\setminus\{0\}$ is a saturated multiplicative system, where $A\subset B$ be a fields. 
\end{pr}

\begin{pr}
	$B[M]\setminus\{0\}$ be a saturated multiplicative system, where $B$ be a field, $M$ be a additive monoid with neutral element $0$. 
\end{pr}

\section{Modules}

In this chapter I will discuss the basic information related to modules.

\begin{pr}
	\label{prpr3}
	Let $R\in\{T, B[X], I(B,A)\}$. Then $$M\in\{T, T_n, T'_n, I(B,A)\}$$ be $R$-modules.
\end{pr}

\begin{proof}
	Easy to check.
\end{proof}

\bigskip

\noindent
{\bf Question:}
What are the smallest assumptions for a sequence below to be an exact sequence in $A+XB[X]$?
\begin{alignat*}{2}
\cdots &\rightarrow A_0+A_1X+\dots + A_{n-1}X^{n-1}+X^nB[X]\rightarrow A+XB[X] &&\rightarrow B[X]\rightarrow \cdots 
\end{alignat*}  

\medskip

\noindent
Consider functions $f\colon A_0+A_1X+\dots + A_{n-1}X^{n-1}+X^nB[X]\to A+XB[X]$ and $g\colon A+XB[X]\to B[X]$, where $A_0, A_1, \dots , A_{n-1}, A, B$ be a domains such that $A_0\not\subset A_i$ for some $i\in\{1, 2, \dots , n-1\}$ and $A, A_0 A_1, A_2, \dots , A_{n-1}\subset B$. Consider $y\in\ker g$, then $g(y)=0$. This implies $y=0=f(0)$ and then $y\in\Im f$.

\medskip

We do not know, whether we can prove $\Im f\subset \ker g$ with the same assumptions.

\begin{lm}
	\begin{itemize}
		\item[(i) ] If $C$ be a subgroup of $A$, then $C+XB[X]$ be a submodule of $B[X]$-module $A+XB[X]$.
		\item[(ii) ] If $C_i$ be a subgroup of $A_i$ for $i=0, 1, \dots , n-1$, then $C_0+C_1X+\dots +C_{n-1}X^{n-1}+X^nB[X]$ be a submodule of $B[X]$-module $A_0+A_1X+\dots +A_{n-1}X^{n-1}+X^nB[X]$.
		\item[(iii) ] If $C$ be a subdomain of $A$ and $D$ be a subdomain of $B$, then $I(D,C)$ be a submodule of $B[X]$-module $I(B,A)$.
		\item[(iv) ] If $N$ be a submonoid of $M$, then $B[N]$ be a submodule of $B$-module $B[M]$, where $B$ be a field.
	\end{itemize}
\end{lm}

\begin{proof}
	Easy to proof.
\end{proof}

\begin{pr}
	\begin{itemize}
		\item[(i) ] If $A$ be a simple $R$-module, then $XB[X]$ and $T$ be unique submodules of $T$.
		\item[(ii) ] For some $i\in\{0, 1, 2, \dots , n-1\}$ consider $A_i$ which be a simple $R$-module. Then $A_0+\dots + A_{i-1}X^{i-1}+C_iX^i+A_{i+1}X^{i+1}+\dots A_{n-1}X^{n-1}+X^nB[X]$ be unique submodules of $T_n$ or $T'_n$, where $C_i\in\{0, A_i\}$. 
		\item[(iii) ] Consider $A_0, A_1, \dots , A_{n-1}\subset B$, where for finitely $i\in\{0, 1, \dots , n-1\}$ assume $A_i$ be simple $R$-modules. Then $C_0+C_1X+\dots C_{n-1}X^{n-1}+X^nB[X]$ be unique submodule of $T_n$ or $T'_n$, where for simple $R$-modules $A_{i_1}, A_{i_2}, \dots , A_{i_k}$ ($i_1, i_2, \dots i_k\in\{0, 1, \dots , n-1\}$) we have $C_{i_1}\in\{0, A_{i_1}\}$, $C_{i_2}\in\{0, A_{i_2}\}, \dots , C_{i_k}\in\{0, A_{i_k}\}$.
		\item[(iv) ] Consider a simple $R$-module $M$ and a field $B$. Then $0, B[M]$ be unique submodules of $R$-module $B[M]$.
	\end{itemize} 
\end{pr}

\begin{proof}
	$(i)$ If $A$ be a simple module then $0$ and $A$ be a unique submodules of $A$. Hence we get $C+XB[X]$, where $C\in\{0, A\}$. Proof of $(ii), (iii), (iv)$ hold similarly to $(i)$. 
\end{proof}

\section{Localizations}

This chapter will tell you about building fractions. The following statements are very well known, but have been included in the structures we are considering.

\begin{tw}
	Let $S$ be a multiplicative subset of $T_n$. There exists a ring $T_{n,S}$ and a homomorphism $w\colon T_n\to T_{n,S}$ which satisfy the following conditions:
	\begin{itemize}
		\item[(i) ] for all $s\in S$, the elements $w(s)$ are invertible in $T_{n,S}$,
		\item[(ii) ] given a homomorphism of composites $f\colon T_n\to T_m$ such that, for every $s\in S$, the elements $f(s)$ are invertible in $T_m$, there exists a unique homomorphism $F\colon T_{n,S}\to T_m$ such that $Fw=f$.
	\end{itemize}
	The ring $T_{n,S}$ is determined by the above conditions up to isomorphism.
\end{tw}

\begin{proof}
	Let us consider the set $T_n\times S$ of all pairs $(g,s)$, where $g\in T_n, s\in S$, and define the operations $+$, $\cdot$ in it by formulae
	$$(g,s)+(g_1,s_1)=(gs_1+g_1s,ss_1), (g,s)\cdot(g_1,s_1)=(gg_1,ss_1).$$
	The pairs $(g,s), (g_1,s_1)\in T_n\times S$ are called equivalent ($=$) if and only if there exists $s_2\in S$ such that $s_2(gs_1-g_1s)=0$. The relation thus defined is a relation of equivalence. The equivalence class of a pair $(g,s)\in T_n\times S$ with respect to the relation $=$ is denoted by $gs^{-1}$ or $g/s$ and called fraction. The set of all fractions $g/s$ with $g\in T_n, s\in S$ forms a ring denoted $T_{n,s}$. Fractions $g/s, g'/s'\in T_{n,s}$ are equal if and only if there exists an element $s''\in S$ such that $s''(rs'-r's)=0$. We define the ring homomorphism $w\colon T_n\to T_{n,S}$ by the formula $w(g)=g/1$ for $g\in T_n$. The ring $T_{n,s}$ and the homomorphism $w$ satisfy conditions $(i)$ and $(ii)$. 
\end{proof}

\begin{cor}
	$T_{n,S}=A_0S+A_1SX+\dots +A_{n-1}SX^{n-1}+X^nBS[X]$.
\end{cor}

\noindent
Similarly, we can construct a fraction module.

\begin{lm}
Let $S$ be a multiplicative subset of $I(B,A)$. Then $I(B,A)_S=I(B_S,A)$ with $B$ be a field.
\end{lm}

\begin{pr}
	Let $M\subset \mathbb{Q}_+$ be a subgroup and let $B$ be a ring. Then $B[M]_0=B_0[M]$.
\end{pr}

\begin{proof}
	Consider $f\in B[M], g\in B[M]\setminus\{0\}$. When we divided $f$ by $g$ we get a sum of elements form $\dfrac{a}{b}X^{m-n}$, where $a\in B, b\in B\setminus\{0\}, m, n\in M$. We see that $m-n\in M$ and $\dfrac{a}{b}\in B_0$. 
\end{proof}

\section{Graded rings and modules}

The last chapter shows the considered structures as graded rings and modules. Later, we can use it to study different properties.

\begin{pr}
	\begin{itemize}
		\item[(i) ] $T$ be a graded ring. 
		\item[(ii) ] $T_n$ be a graded ring.
		\item[(iii) ] $I(B,A)$ be a graded ring.
		\item[(iv) ] If $M$ be a countable monoid and $B$ be a domain, then $B[M]$ be a graded ring. 
	\end{itemize}
\end{pr}

\begin{proof}
	Obvious.
\end{proof}

\begin{ex}
	$T'_n$ is no a graded ring. Consider $T'_3=A_0+A_1X+A_2X^2+X^3B[X]$. Then
	
	\noindent
	$R_1R_3=A_0+A_0A_1X+(A_0A_2+A_1)X^2+(B+A_1A_2)X^3+BX^4\in A_0+XB[X]\neq R_4.$
\end{ex}

\begin{pr}
	\begin{itemize}
		\item[(i) ] $T$ be a graded $T$-module.  
		\item[(ii) ] $T_n$ be a graded $T_n$-module.
		\item[(iii) ] $I(B,A)$ be a graded $R$-module, where $R\in\{I(B,A), T\}$.
		\item[(iv) ] $B[M]$ be a graded $R$-module, where $R\in\{I(B,A), B[M]\}$ where $B$ be a field and $M$ be a monoid.
	\end{itemize}
\end{pr}

\section{Composites with monoid and group coefficients. Monoid rings}
\label{RR1}

The main motivation of this paper is description of composites and monoid domains in the language of commutative algebra. Previous chapters were described but these chapters have been described for the rings (for composites) and for domains (for monoid domains). Then I decided to try described for monoids/groups (for composites) and for rings (for monoid rings). 

\begin{lm}
	\label{l1}
	$T_n$ is a monoid, if components ($A_0, \dots , A_{n-1}, B$) are monoids, $T_n\subset T$ and $T'_n\subset T$.
\end{lm}

Of course Lemma \ref{p1} and Corollary \ref{c1} hold, but in monoid version. 
The same holds true for invertible and nilpotent elements in polynomial composites and monoid rings (Proposition \ref{p2} and Proposition \ref{p3}, proofs are similar). Proposition \ref{prpr1} can also be used for monoid rings.

\medskip

Unfortunately, it is not known whether the Theorem \ref{t2} can occur in the monoid or group version.

\medskip

Similar claims are in the case of irreducible elements (Theorem \ref{t3}, Proposition \ref{x1}). We often come across the concept of atom more in groups or in monoids.

\section{Some examples}

Recall that a ring $R$ satisfies ACCP if each chain of principal ideals of $R$ is stabilize.

\begin{ex}
	In \cite{1} Example 5.1 showed an example of an integral domain $R$ which satisfies ACCP, but whose integral closure does not satisfy ACCP. It mean $R=\mathbb{Z}+X\overline{\mathbb{Z}}[X]$, where $\overline{\mathbb{Z}}$ be the ring of all algebraic integers. $R$ satisfies ACCP. For if not, then there is an infinite properly ascending chain of pricipal ideals of $R$. Since the degrees of the polynomials generating these principal ideals are nonincreasing, the degrees eventually stabilize. The principal ideals in $\overline{\mathbb{Z}}$ generated by the leading coefficients of these polynomials gives an infinite ascending chain $a_1\overline{\mathbb{Z}}\subsetneq a_2\overline{\mathbb{Z}}\subsetneq ...$ where each $a_n/a_{n+1}\in\mathbb{Z}$. Thus all $a_n\in\mathbb{Q}[a_1]$. Let $A=\overline{\mathbb{Z}}\cap\mathbb{Q}[a_1]$. Then $a_1A\subsetneq a_2A\subsetneq\dots \subsetneq A$, a contradiction since $A$ is Dedekind.
\end{ex}

\begin{ex}
	Recall that a domain $R$ is called a half-factorial domain (HFD) if $R$ is atomic and for each nonzero nonunit $x\in R$, $x=x_1\dots x_m=y_1\dots y_n$ where $x_i, y_j$ are all irreducible for $i=1, \dots , m, j=1, \dots , n$, implies that $m=n$. A HFD domain satisfies ACCP.
\end{ex}

\begin{ex}
	If $K_1\subsetneq K_2$ is algebraic extension of fields, then $R=K_1+XK_2[X]$ has integral closure $K_2[X]$, a Euclidean domain. If $[K_2\colon K_1]<\inf$, then $R$ is a Noetherian HFD that is not integrally closed. If $K_1$ is algebraaically closed in $K_2$, then $R$ is an integrally close non-Noetherian HFD. Of course, $R$ satisfies ACCP.
\end{ex}

\begin{ex}
	Let $R=\mathbb{R}+X\mathbb{C}[X]$. So $R$ is a HFD, so has ACCP, then atomic.
\end{ex}

\begin{ex}
	(\cite{x11}) Let $F$ be a field and $T$ the additive submonoid of $\mathbb{Q}^+$ generated by $\{1/3, 1/(2\cdot 5), \dots , 1/(2^jp_j), \dots \}$, where $p_0=3, p_1=5, \dots $ is the sequence of odd primes. Let $R$ be the monoid domain $F[X;T]=F[T]$ and $N=\{f\in R\mid f \mbox{has nonzero constant term}\}$. Then $F[T]_N$ is an atomic domain which does not satisfy ACCP.
\end{ex}

\begin{ex}
	Let $K$ be a field and $T$ the additive submonoid of $\mathbb{Q}^+$ generated by $\{1/2, 1/3, 1/5, \dots , 1/p_j, \dots\}$, where $p_j$ is the $j$th prime. Then the monoid domain $R=K[T]$ satisfies ACCP.
	
	For a $0\neq f=b_1X^{a_1}+\dots b_nX^{a_n}\in R$ with $a_1<\dots <a_n$ and $b_n\neq 0$, write $\beta (f)=a_n$. If ACCP fails, the there is a strictly increasing chain $(f_1)\subset (f_2)\subset \dots $ of principal ideals in $R$. Then each $f_n=f_{n+1}g_{n+1}$ for some nonunit $g_{n+1}\in R$. Hence each $\beta (f_n)=\beta (f_{n+1})+\beta (g_{n+1})$, and each term is positive. Then in $T$, we have $\beta(f_1)>\beta (f_2)>\dots $ with each $\beta (f_n)-\beta (f_{n+1})\in T$, but this is impossible by the above-mentioned unique representation of each nonzero $a\in T$.
\end{ex}

\begin{ex}
	(\cite{0}) Let $K$ be a field, $T=\{q\in\mathbb{Q}\mid q\geqslant 1\}\cup\{0\}$ an additive submonoid of $\mathbb{Q}^+$, and $R=K[T]$ the monoid domain. Then $R_S=K[\mathbb{Q}]$, where $S=\{X^t\mid t\in T\}$, is not atomic since $R_S$ is a GCD-domain, but $R_S$ does not satisfy ACCP.
\end{ex}

\section{Applications of polynomial composites in cryptology}
\label{R5}

Each such polynomial is the sum of the products of the variable and the coefficient. And what if subsequent coefficient sets are appropriate cryptographic systems? Instead of encrypting with one system, we can create one system composed of many systems. Such a cipher is very difficult to break. If the spy detects encryption systems (composite coefficients), then the problem will be to find the right sum and product of such systems.	

Assume that we have two people: Alice und Bob. Alice wants to send a message to Bob. Alice has one composite type $T_n'$ and Bob has another one composite type $T_n'$. 

\medskip

We can build such composite by various encryption systems (even known ones).
Let see note Lemma:

\begin{lm}
	Let $f=a_0+a_1X+\dots + a_{n-1}X^{n-1}+ \sum_{j=n}^{m}a_jX^j$, $g=b_0+b_1X+\dots + b_{n-1}X^{n-1}+ \sum_{j=n}^{m}b_jX^j\in T_n'$, where $a_i, b_i\in A_i$ for $i=1, 2, \dots , n-1$ and $a_j, b_j\in B$ for $j=n, n+1, \dots m$. Then 
	$$fg\in A_0+XB[X].$$  
\end{lm}

Put $A_i, B_j$ $(i, j=0, 1, \dots , n-1)$ be different encryption systems. Then we have $f$ and $g$ are composition of encryption systems. No consider $B$. To improve security, let's fix that $\deg f=n-1, \deg g=n-k$, where $k\in\{2, \dots n-1\}$. And such $f,g$ Alice and Bob agree before the message is sent. 

\medskip

Alice and Bob multiply these composites to form one. 
We have \\
$fg=(A_0+A_1X+\dots A_kX^k)(B_0+B_1X+\dots +B_lX^l)=A_0B_0+(A_0B_1+A_1B_0)X+\dots +A_kB_lX^{k+l}.$

Note that the sum and product of the encryption systems must be defined in the formula above. Definitions we leave Alice and Bob. But in this section we can put $S_iS_j: x\to (x)_{S_i}(x)_{S_j}$ and $S_i+S_j: x\to ((x)_{S_i})_{S_j}$. 

So in the product we encrypt the letter as two letters, the first in the first system and the second in the second system. And in the sum we encrypt the letter using the first system and then the second system. Of course, we can define completely different, at our discretion.

\medskip

Assume that degree of $fg$ is $m$ and text to encrypt consists of more letters then $m+1$. 
Then we divide the text into blocks of length $m + 1$. 
We can assume that $fg(0)$ encrypts the first letter of each block. Expression at $X$ of $fg$ encrypts the second letter of each block, and expression at $X^2$ of $fg$ encrypts the third letter and so on.

\medskip

Now, let's see how to decrypt in this idea.

\medskip

Assume that we have an encrypted message $M_0M_1\dots M_n$. If our key is degree $m$, then we divide message on $m+1$ partition. And every partion divide to two. Every two letters are one letter of message. 

Earlier we define $S_iS_j: x\to (x)_{S_i}(x)_{S_j}$ and $S_i+S_j: x\to ((x)_{S_i})_{S_j}$. Then decryption of two letters $M_lM_{l+1}$ $(l=0, 2, 4, \dots )$ are $M_lM_{l+1}=(M_l)_{S_i}(M_{l+1})_{S_j}=N_{l,l+1}$ (one letter) and $M_l=((M_l)_{S_i})_{S_j}=(N_l)_{ij}$ (one letter).

\begin{ex}

Alice and Bob agree different encryption systems in the center: $A_0, A_1, A_2, B_0, B_1$. Next, Alice has gone far from Bob.

\medskip

We have two compositions: $f=A_0+A_1X+A_2X^2$, $g=B_0+B_1X$. Their key is one composition in the form $fg$ i.e.
$$A_0B_0+(A_0B_1+A_1B_0)X+(A_2B_0+A_1B_1)X^2+A_2B_1X^3.$$

\medskip

The established systems are as follows:

$A_0$ is a Caesar cipher, where the letter is shifted one letter forward;

$A_1$ is a Caesar cipher, where the letter is shifted two letter forward;

$A_2$ is a Caesar cipher, where the letter is shifted three letter forward;

$B_0$ is a Caesar cipher, where the letter is shifted one letter back;

$B_1$ is a Caesar cipher, where the letter is shifted two letter back.

\medskip

Suppose Alice wants to send a message saying 
$$0\quad 2\quad 4\quad 6\quad 8\quad 9\quad 6\quad 5$$

\medskip

The degree of $fg$ is $3$. Hence message divide to $3+1$ partition. So the fourth letter is the same encrypted.

\medskip

Letters $0$ and $8$ encrypt by $A_0B_0$. Then, from definition of $A_0, B_0$, $0$ will be $1$ $9$ (two letters). $8$ will be $9$ $7$. 

\medskip

Letters $2$ and $9$ encrypt by $A_0B_1+A_1B_0$. Then $2$ will be $5$ $9$ and $9$ will be $2$ $6$.

\medskip

Letters $4$ and $6$ encrypt by $A_2B_0+A_1B_1$. Then $4$ will be $9$ $1$ and $6$ will be $1$ $3$. 

\medskip

Letters $6$ and $5$ encrypt by $A_2B_1$. Then $6$ will be $9$ $4$ and $5$ will be $8$ $3$.

\medskip

Bob receives a message from Alice:
$$1\quad 9\quad 5\quad 9\quad 9\quad 1\quad 9\quad 4\quad 9\quad 7\quad 2\quad 6\quad 1\quad 3\quad 8\quad 3$$

\medskip

Now, Bob would like to read the message. Bob sees that message has $16$ letters, so the original text has $8$ letters, because the composition $fg$ has degree $3$ (i.e. $(3+1)2$ letters of original message). Divide message by $8$ letters. 

\medskip

We take the first pairs from each section, i.e. $1$ $9$ and $9$ $7$. Bob uses decryption $(A_0B_0)^{-1}$. So, $1$ will be $0$ by $A_0^{-1}$ and $9$ will be $0$ by $B_0^{-1}$. Hence $1$ $9$ will be $0$. Similarly, $9$ $7$ will be $8$.

\medskip

Next, we take the second pairs from each section, i.e. $5$ $9$ and $2$ $6$. Bob uses decryption $(A_0B_1+A_1B_0)^{-1}$. So, $5$ $9$ will be $2$ and $2$ $6$ will be $9$. 

\medskip

We take next pair, i.e. $9$ $1$ and $1$ $3$. Bob uses decryption $(A_2B_0+A_1B_1)^{-1}$. So, $9$ $1$ will be $4$ and $1$ $3$ will be $6$.

\medskip

Similarly, the last pairs decrypt by $(A_2B_1)^{-1}$. The pair $9$ $4$ will be $6$ and $8$ $3$ will be $5$.

\medskip

After decrypting, Bob received the message:
$$0\quad 2\quad 4\quad 6\quad 8\quad 9\quad 6\quad 5$$

\end{ex}

\section{The concept of using monoid domains in cryptology}

Recall that if $F$ be a field and $M$ be a submonoid of $\mathbb{Q}_+$ then we can construct a monoid domain:

$$F[M]=F[X;M]=\{a_0X^{m_0}+\dots +a_nX^{m_n}\mid a_i\in F, m_i\in M\}.$$

\medskip

Any alphabet of characters creates a finite set. Most ciphers are based on finite sets. But we can have the idea of using the infinite alphabet $\mathbb{A}$, although in reality they can be cyclical sets with an index that would mean a given cycle. For example, A$_0$ - $0$, B$_0$ - $1$, $\dots$ , Z$_0$ - $25$, A$_1$ - $0$, B$_1$ - $1$, $\dots$ , where A$_i$=A, $\dots$, Z$_i$=Z for $i=0, 1, \dots $. We see that this is isomorphic to a monoid $\mathbb{N}_0$ non-negative integers by a formula 

$$f\colon\mathbb{A}\to\mathbb{N}, f(m_i)=i.$$

\medskip

Then we can use a monoid domain by a map 
$$\varphi\colon\mathbb{A}\to F[\mathbb{A}], \varphi(m_0, m_1, \dots , m_n)=a_0X^{m_0}+\dots a_nX^{m_n}.$$

Here, one should think carefully about what a field $F$ should be and think about additional mappings. In contrast, monoid domains can be excellent carriers of characters in the alphabet for monoids. This will make it harder to break any ciphers based on monoids for one simple reason, namely, we don't have inverse properties in a monoid.

\end{document}